\documentclass[10pt]{amsart}
\usepackage{amsmath,amscd,amssymb}
\usepackage[mathscr]{eucal}
\usepackage[all]{xy}
\newcommand{\bZ}{{\mathbb Z}}
\newcommand{\Sp}{\mathop{\mathrm{Sp}}}
\newcommand{\SU}{\mathop{\mathrm{SU}}}
\newcommand{\s}{\mathop{\mathrm{S}}\nolimits}

\title{Homotopy groups of the spaces of self-maps of Lie groups}
\author{Ken-ichi Maruyama}
\address{Chiba University, Chiba, Chiba 263-8522, Japan}
\email{maruyama@faculty.chiba-u.jp}  
\author{Hideaki \=Oshima}
\address{Ibaraki University, Mito, Ibaraki 310-8512, Japan}
\email{ooshima@mx.ibaraki.ac.jp}
\begin{document}

\begin{abstract}
We compute the homotopy groups of the spaces of self maps of Lie groups of rank 2, $\SU(3)$, $\Sp(2)$, and $G_2$. 
We use the cell structures of these Lie groups and the standard methods of homotopy theory.      
\end{abstract}
\maketitle
\newtheorem{prop}{Proposition}[section]
\newtheorem{lemma}[prop]{Lemma}
\newtheorem{theorem}[prop]{Theorem}
\newtheorem{corollary}[prop]{Corollary}
\newtheorem{example}[prop]{Example}
\newtheorem{conjecture}[prop]{Conjecture}
\newtheorem{definition}[prop]{Definition}
\newtheorem{notation}[prop]{Notation}
\numberwithin{equation}{section}
\renewcommand{\baselinestretch}{1.5}
%
\section{Introduction}

For pointed spaces $X$ and $Y$, we let $\mathrm{map}_*(X, Y)$ denote the space of pointed  maps from $X$ to $Y$. We take the trivial map $*$ as a base point of $\mathrm{map}_*(X, Y)$.  
The homotopy groups of function spaces have long been studied in homotopy theory. 
Indeed, if $X = \s^n$, then $\mathrm{map}_*(\s^n, Y)$ coincides with the iterated loop space $\Omega^n Y$. 
Hence the homotopy groups $\pi_n\mathrm{map}_*(\s^n, Y)$ are known by the homotopy groups of $Y$.  
However, even if the number of the cells of $X$ is small, the determination of the group structure of $\pi_n\mathrm{map}_*(X, Y)$ is not easy in general. 

In this paper we study the  homotopy groups of the self maps $\mathrm{map}_*(X, X)$ in the case where $X$ is a compact Lie group of rank 2. 
Precisely, we consider $\SU(3)$, $\Sp(2)$, and $G_2$. 
The homotopy-theoretic structures of these spaces are well known. 
In particular, their homotopy groups are computed in Mimura-Toda \cite{MT}, and Mimura \cite{M}. 
Our results entirely depend on their work. 

The homotopy groups of $\mathrm{map}_*(X, X)$ are closely related to the homotopy groups  of other interesting spaces. 
For instance, we have

(i) We can apply our results to the homotopy groups of the spaces of  self-homotopy equivalences. 
When $X$ is a topological group, all connected components of $\mathrm{map}_*(X, X)$ have the same homotopy type. Hence we have an isomorphism:
$$\pi_n(\mathrm{aut}_*(X), 1_X) \cong \pi_n\mathrm{map}_*(X, X)$$
where $\mathrm{aut}_*(X)$ is  the space of the based maps of $X$ which are homotopy equivalences. 
In \cite{didi}, Didierjean studied the homotopy groups of $\pi_n(\mathrm{aut}_*(X))$ for rank 2 Lie groups by using other methods. 
Our results in this paper extend some of the results in \cite{didi}.  

(ii) Our results in this paper can be used to know the homotopy types of the gauge groups $\mathcal{G}(P)$. 
Generally, for a principal $G$-bundle  $P\to X$,
 $$\mathrm{map}_P(X, BG) \simeq B\mathcal{G}(P)$$
 by Atiyah-Bott \cite{AB}, where  $\mathrm{map}_P(X, BG)$ is a subspace of $f \in \mathrm{map}(X, BG)$ 
such that $f$ is homotopic to the classifying map of $P$. 
There exists a fibration as follows.
$$ G \xrightarrow{\alpha} \mathrm{map}_{*,P}(X, BG) \to B\mathcal{G}(P) \to BG,$$
where $\mathrm{map}_{*,P}(X, BG)=\mathrm{map}_*(X, BG)\cap\mathrm{map}_P(X, BG)$. 
In particular,  when $X = \s^n$, the adjoint of the map $\alpha$ is an element of $\pi_{n-1}\mathrm{map}_*(G, G)$.
  
Finally, we make mention of the homotopy group  $\pi_0\mathrm{map}_*(X, X)$. 
This set is considered as the homotopy classes $[X, X]$, and is a group when $X$ is a topological group. 
In the case that $X$ is a connected Lie group of rank 2, $\pi_0\mathrm{map}_*(X, X)$ are studied in \cite{AOS,KO,MO,O1,O2,O3}.   

Now we state our main results in this paper.\\
 
{\sc Theorem 1.}
\begin{center}
\begin{tabular}{|c|c|c|}
\hline
$n$ & $\pi_n\mathrm{map}_*(\SU(3), \SU(3))$ & $\pi_n\mathrm{map}_*(\Sp(2), \Sp(2))$ \\
\hline
1   & $\mathbb{Z}_3^2$ & $\mathbb{Z}_2^2$ \\
\hline
2 & $\mathbb{Z} \oplus \mathbb{Z}_2 \oplus \mathbb{Z}_3 \oplus \mathbb{Z}_5$ & $\mathbb{Z}_2^3$\\
\hline
3 & $\mathbb{Z}_4 \oplus \mathbb{Z}_8 \oplus \mathbb{Z}_3^2$ & $\mathbb{Z}_2 \oplus \mathbb{Z}_4 \oplus \mathbb{Z}_8  \oplus \mathbb{Z}_5$\\
\hline
4 & $\mathbb{Z}_4 \oplus \mathbb{Z}_3^2 \oplus \mathbb{Z}_5$ &  $\mathbb{Z} \oplus \mathbb{Z}_2 \oplus \mathbb{Z}_{16} \oplus \mathbb{Z}_3 \oplus \mathbb{Z}_5 \oplus \mathbb{Z}_7$\\
\hline
5 & $\mathbb{Z}_2 \oplus A \oplus \mathbb{Z}_3^3 \oplus \mathbb{Z}_5$ & $\mathbb{Z}_2^3$ \\
\hline
6 & $\mathbb{Z}_2 \oplus \mathbb{Z}_4^2 \oplus \mathbb{Z}_3^2 \oplus \mathbb{Z}_7$ & $\mathbb{Z}_2^4$ \\
\hline
7 & $\mathbb{Z}_4 \oplus \mathbb{Z}_8 \oplus \mathbb{Z}_3^2  \oplus \mathbb{Z}_9 \oplus \mathbb{Z}_5^2$ & $\mathbb{Z}_{8} \oplus \mathbb{Z}_{32} 
\oplus \bZ_2 \oplus \bZ_9\oplus \mathbb{Z}_5^3 \oplus \mathbb{Z}_{7}$\\
\hline
8 & $\mathbb{Z}_2 \oplus \mathbb{Z}_4\oplus \bZ_8 \oplus \mathbb{Z}_3^2 \oplus \mathbb{Z}_9 \oplus \mathbb{Z}_7$ & $\mathbb{Z}_2^3 \oplus \mathbb{Z}_8 \oplus \mathbb{Z}_9 \oplus \mathbb{Z}_5 \oplus \mathbb{Z}_7$\\
\hline 
\end{tabular}
\ \vspace{0.3cm} \\
\end{center} 
Here $\mathbb{Z}_n^r $ denotes the direct sum of $r$ copies of $\mathbb{Z}_n$, and  
$A$ is $\bZ_{2}\oplus\bZ_4$ or $\bZ_8$. 
Hamanaka-Kono \cite{HK} proves $A=\bZ_8$. 

For the exceptional Lie group $G_2$ we obtain the following.\\

{\sc Theorem 2.}
$\pi_1\mathrm{map}_*(G_2, G_2) \cong \mathbb{Z}_2 \oplus \mathbb{Z}_2$.

\vspace{0.3cm} 

\noindent
{\bf Acknowledgement}. We thank Professor A. Kono for suggesting the relationship between the gauge group theory and our work. 
We thank also the referee for useful comments. 
For example, our original proof of Proposition 4.4 (4) below has been replaced by more simple one. 
\section{Preliminaries}

As defined in the introduction, $\mathrm{map}_*(X, Y)$ denote the function space of pointed maps from $X$ to $Y$.  
We consider $\mathrm{map}_*(X, Y)$ as a topological space having the compact open topology. 
We denote by $ \pi_n\mathrm{map}_*(X, Y)$ the homotopy group of the component of the trivial map.  
Namely,  $$\pi_n\mathrm{map}_*(X, Y) = \pi_n(\mathrm{map}_*(X, Y), *).$$ 
In this paper we shall identify $\pi_n\mathrm{map}_*(X,Y)$ with $[\Sigma^n X,Y]$ by the adjoint isomorphism, 
where $\Sigma^n X=\s^n\wedge X$. 

Recall that if the following diagram is commutative up to homotopy, then we call $\overline{h}$ an {\it extension} of $h$ 
and $\widetilde{f}$ a {\it coextension} of $f$. 
$$
\xymatrix{
W \ar[dr]^f &{} &{} & \Sigma W \ar[d]_{\widetilde{f}} \ar[dr]^{-\Sigma f} \\
& X \ar[r]^g & Y \ar[r]^i \ar[dr]_h & C_g \ar[d]^{\overline{h}} \ar[r]_q & \Sigma X\\
&{} &{} &Z
}
$$
Here $C_g=Y\cup_g CX$ is the reduced mapping cone of $g$, $i$ is the inclusion, and $q$ is the quotient map. 

We follow Toda's notation \cite{T3} for elements of homotopy groups of spheres. 

As is well-known, we have
\begin{gather*}
\SU(3)=\s^3\cup_{\eta_3}e^5\cup_{\phi}e^8,\quad \pi_4(\s^3)=\mathbb{Z}_2\{\eta_3\};\\
\Sp(2)=\s^3\cup_\omega e^7\cup e^{10},\quad \pi_6(\s^3)=\mathbb{Z}_{12}\{\omega\},\quad \omega=\nu'+\alpha_1(3).
\end{gather*}
Let
$$
\xymatrix{
\s^3 \ar[r]^-{i'} & C_{\eta_3} \ar[r]^-j & \SU(3);\quad \s^3 \ar[r]^-{i'} & C_{\omega} \ar[r]^-j & \Sp(2)
}
$$
be the inclusion maps. Write $i=j\circ i'$. 
Let 
$$
q_3:C_{\eta_3}\to \s^5,\quad q:\SU(3)\to \s^8;\quad q_3:C_\omega\to \s^7,\quad q:\Sp(2)\to \s^{10}
$$
be the quotient maps. 
Let
$$
\xymatrix{
\s^3\ar[r]^-i & \SU(3) \ar[r]^-p & \s^5
};\qquad
\xymatrix{
\s^3\ar[r]^-i & \Sp(2) \ar[r]^-p & \s^7
}
$$
be the canonical fibrations. 
As is well-known, $p\circ j=q_3$. 

\begin{notation} 
Given $x\in[\Sigma^m C_{\eta_3}, \SU(3)]$ (resp. $x\in [\Sigma^m C_\omega, \Sp(2)]$), 
an extension of $x$ to $\Sigma^m \SU(3)$ (resp. $\Sigma^m \Sp(2)$) is denoted by $\overline{x}\in[\Sigma^m \SU(3), \SU(3)]$ 
(resp. $\overline{x}\in[\Sigma^m \Sp(2), \Sp(2)]$), that is, $x=\big(\Sigma^m j\big)^*\overline{x}$. 
Given $z\in[\Sigma^m \s^3, \SU(3)]$ (resp. $z\in[\Sigma^m \s^3, \Sp(2)]$), we denote by 
$\overline{\overline{z}}$ 
an element of $[\Sigma^m\SU(3), \SU(3)]$ (resp. $[\Sigma^m\Sp(2),\Sp(2)]$) such that 
$z=\big(\Sigma^m i\big)^*(\overline{\overline{z}})$. 
$$
\xymatrix{
\Sigma^m C_{\eta_3} \ar[r]^-{\Sigma^m j} \ar@{.>}[dr]^x & \Sigma^m \SU(3) \ar@/_/@{.>} [d]_{\overline{x}} \ar@/^/[d]^{\overline{\overline{z}}}\\
\Sigma^m \s^3 \ar[u]^{\Sigma^m i'} \ar[r]_z & \SU(3)
}
;\qquad
\xymatrix{
\Sigma^m C_{\omega} \ar[r]^-{\Sigma^m j} \ar@{.>}[dr]^x & \Sigma^m \Sp(2) \ar@/_/@{.>} [d]_{\overline{x}} \ar@/^/ [d]^{\overline{\overline{z}}}\\
\Sigma^m \s^3 \ar[u]^{\Sigma^m i'} \ar[r]_z & \Sp(2)
}
$$
\end{notation}

For any abelian group $\Gamma$ and a set of prime numbers $P$, 
let $\Gamma_{(P)}$ be the localization of $\Gamma$ at $P$. 
Given maps $f:X\to Y$ and $g:Y\to Z$, we usually denote their composition by $g\circ f$, 
but sometimes we denote it simply by $gf$.

\section{$\pi_n\mathrm{map}_*(\SU(3), \SU(3))$}

The odd primary components of $[\Sigma^n \SU(3), \SU(3)]$ are easily obtained from the results in \cite{T3}, 
since if $p$ is an odd prime, then
$\SU(3)_{(p)} \simeq \s^3_{(p)} \times \s^5_{(p)}$ (homotopy equivalent).  
Thus 
\begin{equation}\label{oddcomponent}
[\Sigma^n \SU(3), \SU(3)]_{(p)}\cong \pi_{n+3}(\s^3 \times \s^5)_{(p)} \oplus \pi_{n+5}(\s^3 \times \s^5)_{(p)} \oplus \pi_{n+8}(\s^3 \times \s^5)_{(p)}.
\end{equation}
Hence in the rest of this section we calculate $[\Sigma^n\SU(3),\SU(3)]_{(2)}$ for $n\ge 1$. 
We use
$$
\begin{array}{|c|c|c||c|c|c|} \hline
n & \pi_n\SU(3) & \mbox{gen. of 2-comp.} & n & \pi_n\SU(3) & \mbox{gen. of 2-comp.} \\ \hline
1,2,4,7 & 0 & & 12 & \bZ_4\oplus\bZ_{15} & [\sigma'''] \ (2[\sigma''']=i_*\mu_3)\\ \hline
3 & \bZ & i_*\iota_3 & 13 & \bZ_2\oplus\bZ_3 & i_*\varepsilon' \\ \hline
5 & \bZ & [2\iota_5] & 14 & \bZ_4\oplus\bZ_2\oplus\bZ_{21} & [\nu_5^2]\nu_{11},\ i_*\mu' \\ \hline
6 & \bZ_2\oplus\bZ_3 & i_*\nu' & 15 & \bZ_4\oplus\bZ_9 & [2\iota_5]\nu_5\sigma_8 \\ \hline
8 & \bZ_4\oplus\bZ_3 & [2\iota_5]\nu_5& 16 & \bZ_4\oplus\bZ_2\oplus\bZ_{63}\oplus\bZ_3 & [2\iota_5]\zeta_5,\ [\nu_5\overline{\nu}_8] \\ \hline
9 & \bZ_3 & & 17 & \bZ_2\oplus\bZ_2\oplus\bZ_{15} & [\nu_5]\nu_{11}^2,\ [\nu_5\eta_8\varepsilon_9] \\ \hline
10 & \bZ_2\oplus\bZ_{15}& [\nu_5\eta_8^2] & 18 & \bZ_2\oplus\bZ_2\oplus\bZ_{15}\oplus\bZ_3  & i_*\overline{\varepsilon}_3,\  [\nu_5\eta_8\mu_9] \\ \hline
11 & \bZ_4 & [\nu_5^2] \ (2[\nu_5^2]=i_*\varepsilon_3) & 19 & \bZ_4\oplus\bZ_2\oplus\bZ_2\oplus\bZ_3^2 & [\sigma''']\sigma_{12},\  [\nu_5\overline{\nu}_8]\nu_{16} \\ \hline
\end{array}
$$
\begin{eqnarray*}
\text{Table 1 :}\ \pi_n(\SU(3))
\end{eqnarray*}
This is contained in \cite{MT} with the following notation: 
$[x] \in \pi_n(\SU(3))$ denotes an element such that $p_*[x] = x$. 

Fist we prove $[\Sigma\SU(3),\SU(3)]_{(2)}=0$. 
By Table 1, we have the following exact sequence.
$$
\begin{CD}
0@>(\Sigma q)^*>>[\Sigma \SU(3), \SU(3)]_{(2)}@>(\Sigma j)^*>>[\s^4\cup_{\eta_4}e^6, \SU(3)]_{(2)}
\end{CD}
$$
It suffices for our purpose to prove 
\begin{equation}\label{eta4-1}
[\s^4\cup_{\eta_4}e^6,\SU(3)]_{(2)}=0.
\end{equation} 
By Table 1 we have the following exact sequence. 
\begin{equation}\label{eta4}
\begin{CD}
\quad \bZ_{(2)}\big\{[2\iota_5]\big\}@>\eta_5^*>>\bZ_2\{i_*\nu'\}@>(\Sigma q_3)^*>>[\s^4\cup_{\eta_4}e^6,\SU(3)]_{(2)}@>(\Sigma i')^*>>0.
\end{CD}
\end{equation}
We use the following theorem \cite[Theorem 2.1]{MT}.

\begin{theorem}[\cite{MT}]\label{MT}
Let $F\overset{i}{\to} X\overset{p}{\to} B$ be a fibration, and $\partial : \pi_n(B) \to \pi_{n-1}(F)$ the boundary operator. 
Assume that $\alpha\in\pi_{m+1}(B),\ \beta\in\pi_l(\s^m)$ and $\gamma\in\pi_k(\s^l)$ satisfying $\partial\alpha\circ\beta=0$ and $\beta\circ\gamma=0$. 
For an arbitrary element $\delta\in\{\partial\alpha,\beta,\gamma\}\subset\pi_{k+1}(F)$, 
there exists an element $\epsilon\in\pi_{l+1}(X)$ such that 
$p_*\epsilon=\alpha\circ\Sigma\beta$ and $i_*\delta=\epsilon\circ\Sigma\gamma$.
\end{theorem}

We apply this theorem to the fibration $\s^3\overset{i}{\to} \SU(3)\overset{p}{\to} \s^5$ by taking 
$$
\alpha=\iota_5,\quad \beta=2\iota_4,\quad \gamma=\eta_4,\quad k=5, \quad l=m=4.
$$
Indeed this case can be applied, since $\beta\circ\gamma=0$ and $\partial\alpha=\eta_3$ 
so that $\partial\alpha\circ\beta=0$. 
It follows that for any $\delta\in\{\partial\alpha,\beta,\gamma\}$ there exists $\epsilon\in\pi_5(\SU(3))$ such that 
$$
p_*\epsilon=\alpha\circ\Sigma\beta=2\iota_5, \quad i_*\delta=\epsilon\circ\Sigma\gamma.
$$
In particular we have $\epsilon=[2\iota_5]$. 
Since $\{\eta_3, 2\iota_4, \eta_4\}=\{\nu',-\nu'\}$ by \cite[(5.4)]{T3}, we then have 
\begin{equation}\label{pi6}
i_*\nu'=[2\iota_5]\circ\eta_5=\eta_5^*[2\iota_5].
\end{equation} 
Hence by (\ref{eta4}) we have (\ref{eta4-1}) as desired. 

In order to calculate $[\Sigma^n \SU(3), \SU(3)]_{(2)}$ for $n\ge 2$, we recall a result of Browder-Spanier 
\cite{bs} that the attaching map of the top cell of an $H$-space is stably trivial. 
Hence
\begin{equation}\label{stabletrivia} 
\Sigma^3 \SU(3) \simeq \s^6 \cup_{\eta_{6}} e^8 \vee \s^{11}.
\end{equation}
More precisely, we can prove 
$$
\Sigma\phi = \Sigma i' \circ\nu_4\circ\eta_7. 
$$
We do not use this equality in this paper. 
So we omit its proof. 
We have

\begin{lemma}\label{su-iso} 
$[\Sigma^n \SU(3), \SU(3)] \cong \pi_{8+n}(\SU(3)) \oplus [C_{\eta_{3+n}}, \SU(3)]$ for $n\ge 2$. 
\end{lemma}
\begin{proof}
If $n \geq 3$, then the result follows from (\ref{stabletrivia}). 
For $n = 2$, we have 
$$[\Sigma^2 \SU(3), \SU(3)] \cong [\Sigma^3 \SU(3), B\SU(3)]$$ 
and the lemma follows also from (\ref{stabletrivia}).
\end{proof}

Hence it suffices for our purpose to determine $[C_{\eta_{3+n}}, \SU(3)]_{(2)}$ for $n\ge 2$.

The generators of the 2-components of $[\Sigma^n \SU(3), \SU(3)]$ 
are as follows.
\newcommand{\tabtopsp}[1]{\vbox{\vbox to#1{}\vbox to10pt{}}}
\begin{center}
\begin{tabular}{|c|c|c|}
\hline
$n$ & \text{$2$-{\it components}} & \text{\it generators} \\
\hline\tabtopsp{1mm}
1   & $0$ & {} \\[2mm]
\hline\tabtopsp{0.2cm}
2 & $\mathbb{Z} \oplus \mathbb{Z}_2$  & $\overline{\overline{2[2\iota_5]}},\  (\Sigma^2q)^*[\nu_5\eta_8^2]$\\[2mm]
\hline\tabtopsp{0.2cm}
3 & $\mathbb{Z}_4 \oplus \mathbb{Z}_8$ & $(\Sigma^3q)^*[\nu_5^2]$,\ $\overline{\overline{i_*\nu'}}$ \\[2mm]
\hline\tabtopsp{0.2cm}
4 & $\mathbb{Z}_4$ &  $(\Sigma^4 q)^*[\sigma''']$\\[2mm]
\hline\tabtopsp{0.2cm}
5 & $\mathbb{Z}_2 \oplus \mathbb{Z}_8 $ & $(\Sigma^5q)^*i_*\varepsilon',\ \overline{\overline{[2\iota_5]\circ\nu_5}} $ \\[2mm]
\hline\tabtopsp{0.2cm}
6 & $\mathbb{Z}_2 \oplus  \mathbb{Z}_4 \oplus \mathbb{Z}_4 $ & $(\Sigma^6q)^*i_*\mu',\ (\Sigma^6q)^*([\nu_5^2]\circ\nu_{11}),\    \overline{{\Sigma^6q_3}^*[\nu_5^2]} $ \\[2mm]
\hline\tabtopsp{0.2cm}
7 & $\mathbb{Z}_4 \oplus \mathbb{Z}_8 $ & $(\Sigma^7q)^*([2\iota_5]\circ\nu_5\sigma_8),\  \overline{\overline{[\nu_5\eta_8^2]}}, $\\[2mm]
\hline\tabtopsp{0.2cm}
8 & $\mathbb{Z}_2 \oplus \mathbb{Z}_4 \oplus \bZ_8 $ & $(\Sigma^8q)^*[\nu_5\bar{\nu}_8],\ (\Sigma^8q)^*([2\iota_5]\circ\zeta_5), \ \overline{\overline{[\nu_5^2]}}$\\[2mm]
\hline
\end{tabular}
\end{center}
\begin{eqnarray*}
\text{Table 2 : $2$-components of}\ [\Sigma^n \SU(3), \SU(3)]
\end{eqnarray*}

\vspace{0.3cm}

\subsection{$[C_{\eta_5}, \SU(3)]$} 

By Table 1, we have the following exact sequence.
$$
\begin{CD}
0 @>>>[\s^5\cup_{\eta_5}e^7, \SU(3)]@>>>\bZ\big\{[2\iota_5]\big\}@>\eta_5^*>>\bZ_2\{i_*\nu'\}\oplus\bZ_3
\end{CD}
$$
Hence by (\ref{pi6}) we have $[C_{\eta_5}, \SU(3)]=\mathbb{Z}\big\{\overline{2[2\iota_5]}\big\}$.  
Thus we obtain
$$
[\Sigma^2 \SU(3), \SU(3)]=\mathbb{Z}\Big\{\overline{\overline{2[2\iota_5]}}\Big\}
\oplus\mathbb{Z}_2\big\{\big(\Sigma^2 q\big)^*[\nu_5\eta_8^2]\big\}\oplus\mathbb{Z}_{15}.
$$

\vspace{0.3cm}

\subsection{$[C_{\eta_6}, \SU(3)]_{(2)}$}

By \cite{T3} and Table 1, we have the following commutative diagram with exact rows and columns.
$$
\begin{CD}
\bZ_2\{\nu'\eta_6\} @>\eta_7^*>\cong> \bZ_2\{\nu'\eta_6^2\}@>>>[C_{\eta_6}, \s^3]_{(2)}@>>>\bZ_4\{\nu'\}@>\eta_6^*>>\bZ_2\{\nu'\eta_6\}\\
@VVV @VVV @VV i_*V @VV i_*V @VVV \\
0 @>>>\bZ_4\big\{[2\iota_5]\nu_5\big\} @>(\Sigma^3q_3)^*>>[C_{\eta_6}, \SU(3)]_{(2)} @>(\Sigma^3i')^*>>\bZ_2\{i_*\nu'\} @>>> 0\\
@VVV @VV p_*V @VV p_*V @VVV @VVV \\
\bZ_2\{\eta_5^2\}@>\eta_7^*>> \bZ_8\{\nu_5\}@>(\Sigma^3q_3)^*>>[C_{\eta_6},\s^5]_{(2)}@>>>\bZ_2\{\eta_5\}@>\eta_6^*>\cong>\bZ_2\{\eta_5^2\}
\end{CD}
$$
By the first and third rows, we have the following results (\cite[Propositions 3.3 and 3.1]{KMNST}):
\begin{equation}\label{sphere}
[C_{\eta_6}, \s^3]_{(2)}=\mathbb{Z}_2\big\{\overline{2\nu'}\big\},\quad [C_{\eta_6},\s^5]_{(2)}=\mathbb{Z}_4\big\{(\Sigma^3q_3)^*\nu_5\big\}.
\end{equation}
By the second row, the order of $[C_{\eta_6}, \SU(3)]_{(2)}$ is $8$. 
Hence the middle column is short exact by (\ref{sphere}). 
Since 
$$
p_*(\Sigma^3q_3)^*([2\iota_5]\circ\nu_5)=(\Sigma^3q_3)^*p_*([2\iota_5]\circ\nu_5)=2(\Sigma^3q_3)^*\nu_5,
$$ 
we have $[C_{\eta_6}, \SU(3)]_{(2)}\not\cong \mathbb{Z}_4\oplus\mathbb{Z}_2$. 
Hence $[C_{\eta_6},\SU(3)]_{(2)}=\mathbb{Z}_8\big\{\overline{i_*\nu'}\big\}$.

\vspace{0.3cm}
\subsection{$[C_{\eta_7}, \SU(3)]_{(2)}$}

By Table 1, we easily see that $[C_{\eta_7}, \SU(3)]_{(2)} = 0$. 

\vspace{0.3cm}
\subsection{$[C_{\eta_8}, \SU(3)]_{(2)}$} 

By Table 1, we have the following exact sequence:
$$
\begin{CD}
0@>>>\mathbb{Z}_2\big\{[\nu_5\eta_8^2]\big\}@>(\Sigma^5q_3)^*>>[C_{\eta_8},\SU(3)]_{(2)}@>(\Sigma^5i')^*>>\mathbb{Z}_4\big\{[2\iota_5]\circ\nu_5\big\}@>>>0
\end{CD}
$$  
This does not split as shown by Hamanaka-Kono \cite{HK}. 
Hence 
$$
[C_{\eta_8},\SU(3)]_{(2)}=\mathbb{Z}_8\big\{\overline{[2\iota_5]\circ\nu_5}\big\}.
$$ 
\vspace{0.3cm}
\subsection{$[C_{\eta_9}, \SU(3)]_{(2)}$}

By Table 1, we have the following exact sequence:
$$
\begin{CD}
\bZ_2\big\{[\nu_5\eta_8^2]\big\} @>\eta_{10}^*>> \bZ_4\big\{[\nu_5^2]\big\} @>(\Sigma^6q_3)^*>> [C_{\eta_{9}}, \SU(3)]_{(2)}  @>>>  0
\end{CD}
$$
Thus $\eta_{10}^*[\nu_5\eta_8^2]$ is $0$ or $2[\nu_5^2]$. 
To induce a contradiction, assume $\eta_{10}^*[\nu_5\eta_8^2]=2[\nu_5^2]$. 
Then 
$2\big([\nu_5^2]\circ\nu_{11}\big)=\big(2[\nu_5^2]\big)\circ\nu_{11}=[\nu_5\eta_8^2]\circ\eta_{10}\circ\nu_{11}=0$ 
since $\eta_{10}\circ\nu_{11}=0$ by \cite{T3}. 
This contradicts the fact that the order of $[\nu_5^2]\circ\nu_{11}$ is $4$. 
Hence 
\begin{equation}\label{nueta^2}
[\nu_5\eta_8^2]\circ\eta_{10}=0
\end{equation}
so that 
$$
[C_{\eta_9}, \SU(3)]_{(2)}=\mathbb{Z}_4\big\{(\Sigma^6q_3)^*[\nu_5^2]\big\}.
$$

\vspace{0.3cm}
\subsection{$[C_{\eta_{10}}, \SU(3)]_{(2)}$} 

The purpose of this subsection is to prove 
\begin{equation}\label{pi-70}
[C_{\eta_{10}}, \SU(3)]_{(2)}=\mathbb{Z}_8\big\{\overline{[\nu_5\eta_8^2]}\big\}.
\end{equation}
By \cite{T3}, Table 1 and (\ref{nueta^2}), we have the following commutative diagram with exact rows and columns:
\begin{equation}\label{pi-69}
\begin{CD}
\mathbb{Z}_2\{\varepsilon_3\}@>\eta_{11}^*>>\mathbb{Z}_2^2\{\varepsilon_3\eta_{11}, \mu_3\}
@>(\Sigma^7q_3)^*>>[C_{\eta_{10}},\s^3]_{(2)}@>>>0 @.\\
@VVV @VV i_*V @VV i_*V @. @.\\
\mathbb{Z}_4\big\{[\nu_5^2]\big\} @>\eta_{11}^* >> \mathbb{Z}_4\big\{[\sigma''']\big\} @>(\Sigma^7q_3)^* >> [C_{\eta_{10}},\SU(3)]_{(2)} 
@>(\Sigma^7i')^*>>\mathbb{Z}_2\big\{[\nu_5\eta_8^2]\big\}@>\eta_{10}^*>>0 \\
@VVV @VVV @VV p_*V @V\cong V p_*V @.\\
\mathbb{Z}_2\{\nu_5^2\}@>\eta_{11}^*=0>>\mathbb{Z}_2\{\sigma'''\}@>(\Sigma^7q_3)^*>>[C_{\eta_{10}},\s^5]_{(2)}
@>(\Sigma^7 i')^*>> \mathbb{Z}_2\{\nu_5\eta_8^2\} @>\eta_{10}^*>> 0
\end{CD}
\end{equation}
By the first row, we have the following result (\cite[Proposition 3.7]{KMNST}):
\begin{equation}
[C_{\eta_{10}},\s^3]_{(2)} = \mathbb{Z}_2\big\{ (\Sigma^7 q_3)^* \mu_3 \big\}.
\end{equation}
We need

\begin{prop}\label{eta10}
\begin{enumerate}
\item $[\nu_5^2]\circ\eta_{11}=0$.
\item $($\cite[Proposition 3.5]{KMNST}$)$\quad $[C_{\eta_{10}},\s^5]_{(2)}=\mathbb{Z}_4\big\{\overline{\nu_5\eta_8^2}\big\}$.
\end{enumerate}
\end{prop}

Before proving this proposition, we prove (\ref{pi-70}) by using it. 
By Proposition \ref{eta10}, we have the following commutative diagram with exact rows and columns. 
$$
\begin{CD}
@. \mathbb{Z}_2\{\mu_3\}@>(\Sigma^7q_3)^*>\cong> \mathbb{Z}_2\big\{(\Sigma^7q_3)^*\mu_3\big\} @. @.\\
@. @VV i_*V @VV i_*V @. @.\\
0@>>>\mathbb{Z}_4\big\{[\sigma''']\big\}@>(\Sigma^7q_3)^*>>[C_{\eta_{10}},\SU(3)]_{(2)}@>(\Sigma^7 i')^*>>
\mathbb{Z}_2\big\{[\nu_5\eta_8^2]\big\}@>>>0\\
@. @VV p_*V @VV p_*V @V\cong V p_*V @.\\
0@>>>\mathbb{Z}_2\{\sigma'''\}@>(\Sigma^7q_3)^*>>\mathbb{Z}_4\Big\{\overline{\nu_5\eta_8^2}\Big\}
@>(\Sigma^7 i')^*>>\mathbb{Z}_2\{\nu_5\eta_8^2\}@>>>0
\end{CD}
$$
Hence $[C_{\eta_{10}},\SU(3)]_{(2)}$ is isomorphic to $\mathbb{Z}_8$ or $\mathbb{Z}_4\oplus\mathbb{Z}_2$. 
To induce a contradiction, assume it is $\mathbb{Z}_4\oplus\mathbb{Z}_2$. 
Then 
$$
[C_{\eta_{10}},\SU(3)]_{(2)}=\mathbb{Z}_4\Big\{\overline{[\nu_5\eta_8^2]}\Big\}
\oplus\mathbb{Z}_2\Big\{\overline{[\nu_5\eta_8^2]}-(\Sigma^7q_3)^*[\sigma''']\Big\}
$$
since $p_*\overline{[\nu_5\eta_8^2]}$ generates $[C_{\eta_{10}},\s^5]_{(2)}$. 
We have $i_*(\Sigma^7q_3)^*\mu_3=2(\Sigma^7q_3)^*[\sigma''']=2\,\overline{[\nu_5\eta_8^2]}$. 
Hence the cokernel of the second $i_*$ which is isomorphic to $[C_{\eta_{10}},\s^5]_{(2)}$ is 
$\mathbb{Z}_2\oplus\mathbb{Z}_2$. 
This contradicts Proposition \ref{eta10} (2).  
Therefore we obtain (\ref{pi-70}).

\medskip

\noindent{\it Proof of Proposition \ref{eta10}. }
The assertion (2) is proved in \cite[Proposition 3.5 (4)]{KMNST}. 
We prove (1) as follows. 
Since $\eta_{11}$ is of order $2$, $[\nu_5^2]\circ\eta_{11}$ is $0$ or $2[\sigma''']$. 
To induce a contradiction, assume $[\nu_5^2]\circ\eta_{11}=2[\sigma''']$. 
Then, by \cite[Lemma 6.4]{T3} and Table 1, we have
\begin{equation}\label{pi-7}
[\nu_5^2]\circ\sigma_{11}\circ\eta_{18}=[\nu_5^2]\circ\eta_{11}\circ\sigma_{12}=2([\sigma''']\circ\sigma_{12})\ne 0.
\end{equation}
By Table 1, we can write $[\nu_5^2]\circ\sigma_{11}=a\cdot i_*\overline{\varepsilon}_3+ b\cdot[\nu_5\eta_8\mu_9]$ $(a,b\in\bZ)$. 
Then 
$$
\nu_5^2\sigma_{11}=p_*([\nu_5^2]\circ\sigma_{11})=b\cdot \nu_5\eta_8\mu_9.
$$
By \cite[(7.19)]{T3}, $\sigma'\nu_{14}=x\cdot \nu_7\sigma_{10}$ with $x$ odd. 
Hence 
$$
\nu_5\circ\Sigma\sigma'\circ\nu_{15}=\nu_5\circ x\cdot\nu_8\circ\sigma_{11}=\nu_5^2\sigma_{11}.
$$
On the other hand, $\nu_5\circ\Sigma\sigma'=2(\nu_5\sigma_8)$ by \cite[(7.16)]{T3}. 
Hence $\nu_5\circ\Sigma\sigma'\circ\nu_{15}=0$, since $2\pi_{18}(\s^5)_{(2)}=0$ by \cite{T3}. 
Thus $\nu_5^2\sigma_{11}=0$ so that $b$ is even and $[\nu_5^2]\circ\sigma_{11}=a\cdot i_*\overline{\varepsilon}_3$.
We then have
$$
[\nu_5^2]\circ\sigma_{11}\circ\eta_{18}=a\cdot i_*\big(\overline{\varepsilon}_3\eta_{18}\big)
=a\cdot i_*(\eta_3\overline{\varepsilon}_4)=a\cdot (i_*\eta_3\circ\overline{\varepsilon}_4)=0,
$$
since $i_*\eta_3\in\pi_4(\SU(3))=0$. 
This contradicts (\ref{pi-7}). 
Therefore $[\nu_5^2]\circ\eta_{11}=0$. 
\qed

\vspace{0.3cm}
\subsection{$[C_{\eta_{11}}, \SU(3)]_{(2)}$}

By Table 1 and Proposition \ref{eta10} (1), we have the following commutative diagram with exact rows and columns: 
\begin{equation}\label{eta11-0}
\begin{CD}
\mathbb{Z}_4\big\{[\sigma''']\big\} @>\eta_{12}^*>>\mathbb{Z}_2\{ i_*\varepsilon'\} @>(\Sigma^8q_3)^*>> 
[C_{\eta_{11}}, \SU(3)]_{(2)} @>(\Sigma^8i')^*>> \mathbb{Z}_4\big\{[\nu_5^2]\big\} @>>>0\\
@VVV @VV p_*V @VV p_*V @VVV @.\\
\bZ_2\{\sigma'''\}@>\eta_{12}^*>>\bZ_2\{\varepsilon_5\}@>(\Sigma^8 q_3)^*>>[C_{\eta_{11}},\s^5]@>(\Sigma^8 i')^*>>\bZ_2\{\nu_5^2\}@>>>0\\
@VVV @VV\partial V @VV\partial V @VVV @.\\
\bZ_2\{\varepsilon_3\}@>\eta_{11}^*>>\bZ_2^2\{\mu_3, \eta_3\varepsilon_4\}@>(\Sigma^7q_3)^*>>[C_{\eta_{10}},\s^3]_{(2)}@>>>0 @.
\end{CD}
\end{equation}
The purpose of this subsection is to prove 
\begin{equation}\label{eta11-1}
[C_{\eta_{11}},\SU(3)]_{(2)}=\bZ_8\big\{\overline{[\nu_5^2]}\big\},\quad 4\cdot\overline{[\nu_5^2]}=(\Sigma^8q_3)^*i_*\varepsilon'.
\end{equation}

We need two lemmas. 

\begin{lemma}\label{eta12}
\begin{enumerate}
\item $[\sigma''']\circ\eta_{12}=0$.
\item $($\cite[Proposition 3.6]{KMNST}$)$\quad $[C_{\eta_{11}},\s^5]=\bZ_4\big\{p_*\overline{[\nu_5^2]}\big\},\ 
2\cdot p_*\overline{[\nu_5^2]}=(\Sigma^8 q_3)^*\varepsilon_5$.
\end{enumerate}
\end{lemma}
\begin{proof}
Consider the following commutative diagram. 
$$
\begin{CD}
\pi_{12}(\SU(3)) @>{i_{3,4}}_*>\cong> \pi_{12}(\SU(4)) @>{i_{4,5}}_*>> \pi_{12}(\SU(5))\\
@VV\eta_{12}^*V @VV\eta_{12}^*V @VV\eta_{12}^*V \\
\pi_{13}(\SU(3))_{(2)} @>{i_{3,4}}_*>> \pi_{13}(\SU(4)) @>{i_{4,5}}_*>\cong> \pi_{13}(\SU(5))
\end{CD}
$$
Here $i_{k,l}:\SU(k)\to\SU(l)$ is the inclusion map. 
Recall from \cite[Theorem 4.4]{T2} that $\pi_{12}(\SU(5))=\bZ_8\oplus\bZ_{45}$. 
Then the first ${i_{3,4}}_*$ is bijective and the second ${i_{3,4}}_*$ is injective by \cite{MT}. 
Since $\pi_{13}(\s^9)=\pi_{14}(\s^9)=0$ by \cite{T3}, the first ${i_{4,5}}_*$ is injective and the second ${i_{4,5}}_*$ is bijective. 
Let $g$ denote a generator of the $2$-primary part of $\pi_{12}(\SU(5))$ satisfying ${i_{3,5}}_*[\sigma''']=2g$. 
Then 
$$
{i_{3,5}}_*\eta_{12}^*[\sigma''']=\eta_{12}^*{i_{3,5}}_*[\sigma''']=\eta_{12}^*(2g)=g\circ2\eta_{12}=0.
$$
Hence $\eta_{12}^*[\sigma''']=0$ and we obtain (1). 

Since no precise proof of (2) is in \cite{KMNST}, we give a proof of (2). 
We firstly claim that the second $p_*$ of (\ref{eta11-0}) is surjective, that is, the second $\partial$ of (\ref{eta11-0}) is trivial. 
We have
$$
\partial\varepsilon_5=\partial\iota_5\circ\varepsilon_4=\eta_3\varepsilon_4=\varepsilon_3\eta_{11}=\eta_{11}^*\varepsilon_3
$$
so that 
$$
\partial(\Sigma^8q_3)^*\varepsilon_5=(\Sigma^7 q_3)^*\partial\varepsilon_5=(\Sigma^7q_3)^*\eta_{11}^*\varepsilon_3=0.
$$
Of course $\partial p_*\overline{[\nu_5^2]}=0$. 
Hence the second $\partial$ of (\ref{eta11-0}) is trivial, since $[C_{\eta_{11}},\s^5]$ is generated by $(\Sigma^8q_3)^*\varepsilon_5$ and 
$p_*\overline{[\nu_5^2]}$. 

By \cite[(7.4)]{T3}, $\sigma'''\eta_{12}=0$. 
Hence, by the second row of (\ref{eta11-0}), the order of  $[C_{\eta_{11}},\s^5]$ is $4$. 
To induce a contradiction, assume $[C_{\eta_{11}},\s^5]\cong\bZ_2^2$, that is, 
$[C_{\eta_{11}},\s^5]=\bZ_2^2\big\{(\Sigma^8q_3)^*\varepsilon_5, p_*\overline{[\nu_5^2]}\big\}$. 
Then the surjectivity of $p_*:[C_{\eta_{11}},\SU(3)]_{(2)}\to[C_{\eta_{11}},\s^5]$ implies that 
$[C_{\eta_{11}},\SU(3)]_{(2)}$ is generated by at least two elements, that is, it must be that 
$[C_{\eta_{11}},\SU(3)]_{(2)}=\bZ_2\big\{(\Sigma^8q_3)^*i_*\varepsilon'\big\}\oplus\bZ_4\big\{\overline{[\nu_5^2]}\big\}$. 
But this is impossible, since $p_*(\Sigma^8q_3)^*i_*\varepsilon'=(\Sigma^8q_3)^*p_*i_*\varepsilon'=0$. 
Therefore $[C_{\eta_{11}},\s^5]=\bZ_4\big\{p_*\overline{[\nu_5^2]}\big\}$ with 
$2\cdot p_*\overline{[\nu_5^2]}=(\Sigma^8q_3)^*\varepsilon_5$. 
\end{proof}

We use the following fibration:
$$
\begin{CD}
\SU(3)@>\hat{i}>>G_2@>\hat{p}>>\s^6
\end{CD}
$$
We use notations and results of \cite{M} freely. 
By \cite{T3,M} and Table 1, we have the following commutative diagram with exact rows and columns where all groups are localized at $2$:
\begin{equation}\label{eta11-2}
\xymatrix{
  &\bZ_8\big\{\langle\overline{\nu}_6+\varepsilon_6\rangle\big\}\oplus\bZ_2\big\{\hat{i}_*[\nu_5^2]\nu_{11}\big\} \ar[r]^-{(\Sigma^9q_3)^*}_-{\cong} \ar[d]^-{\hat{p}_*} &[C_{\eta_{12}},G_2] \ar[d]^-{\hat{p}_*}\\
\bZ_4\{\sigma''\} \ar[r]^-{\eta_{13}^*} \ar[d]^-{\partial} &\bZ_8\{\overline{\nu}_6\}\oplus\bZ_2\{\varepsilon_6\} \ar[r]^-{(\Sigma^9q_3)^*} \ar[d]^-{\partial} &[C_{\eta_{12}},\s^6] \ar@{->>}[r]^-{(\Sigma^9 i')^*} \ar[d]^-{\partial} &\bZ_2\{\nu_6^2\} \ar[d]^-{\partial} \\
\bZ_4\big\{[\sigma''']\big\}\ar[r]^-{\eta_{12}^*=0} &\bZ_2\{i_*\varepsilon'\} \ar[r]^-{(\Sigma^8q_2)^*} \ar[d] &[C_{\eta_{11}},\SU(3)] \ar@{->>}[r]^-{(\Sigma^8i')^*} \ar[d]^-{\hat{i}_*} &\bZ_4\big\{[\nu_5^2]\big\} \ar[d]^-{\hat{i}_*}\\
 &0 \ar[r] &[C_{\eta_{11}},G_2]\ar[r]^-{(\Sigma^8i')^*}_-{\cong} &\bZ_2\big\{\hat{i}_*[\nu_5^2]\big\}\oplus\bZ_{(2)} 
}
\end{equation}
Here we have used results of \cite{M} that $\pi_{12}(G_2)=\pi_{13}(G_2)=0$. 
We need

\begin{lemma}\label{eta11-3}
\begin{enumerate}
\item $($\cite[Proposition 6.3]{M}$)$\quad $\partial\overline{\nu}_6=\partial\varepsilon_6=i_*\varepsilon'.$
\item $($\cite[Proposition 3.6]{KMNST}$)$\quad $[C_{\eta_{12}},\s^6]_{(2)}=\bZ_4\big\{(\Sigma^9q_3)^*\overline{\nu}_6\big\}\oplus\bZ_4\big\{\Sigma p_*\overline{[\nu_5^2]}\big\}$\  and\newline $2\cdot\Sigma p_*\overline{[\nu_5^2]}=(\Sigma^9q_3)^*\varepsilon_6.$
\end{enumerate}
\end{lemma}
\begin{proof} 
We give a proof of (2), because our notations are different from ones in \cite{KMNST}. 
Consider the following commutative diagram with exact rows:
$$
\xymatrix{
\bZ_2\{\sigma'''\} \ar[r]^-{\eta_{12}^*=0}  &\bZ_2\{\varepsilon_5\} \ar[r]^-{(\Sigma^8q_3)^*} \ar[d]^-{\Sigma} &[C_{\eta_{11}},\s^5]_{(2)}\ar[r]^-{(\Sigma^8i')^*} \ar[d]^-{\Sigma} &\bZ_2\{\nu_5^2\} \ar[r] \ar[d]^-{\Sigma}_-{\cong} & 0\\
\bZ_4\{\sigma''\}\ar[r]^-{\eta_{13}^*}&\bZ_8\{\overline{\nu}_6\}\oplus\bZ_2\{\varepsilon_6\}\ar[r]^-{(\Sigma^9q_3)^*} &[C_{\eta_{12}},\s^6]_{(2)}\ar[r]^-{(\Sigma^9i')^*}&\bZ_2\{\nu_6^2\} \ar[r] & 0
}
$$
By Lemma \ref{eta12} (2), we have
\begin{equation}\label{eta11-7}
2\Sigma p_*\overline{[\nu_5^2]}=(\Sigma^9q_3)^*\Sigma\varepsilon_5=(\Sigma^9q_3)^*\varepsilon_6.
\end{equation}
We have $\eta_{13}^*\sigma''=4\cdot\overline{\nu}_6$ by \cite[(7.4)]{T3} so that we have the following short exact sequence:
$$
\begin{CD}
0@>>>\bZ_4\big\{(\Sigma^9q_3)^*\overline{\nu}_6\big\}\oplus\bZ_2\big\{(\Sigma^9q_3)^*\varepsilon_6\big\}@>>>[C_{\eta_{12}},\s^6]_{(2)} 
@>(\Sigma^9i')^*>>\bZ_2\{\nu_6^2\}@>>>0
\end{CD}
$$
Thus the order of $\Sigma p_*\overline{[\nu_5^2]}$ is $4$ by (\ref{eta11-7}), and 
we obtain (2) by the above exact sequence, since $(\Sigma^9 i')^*\Sigma p_*\overline{[\nu_5^2]}=\nu_6^2$.
\end{proof}

\medskip

\noindent{\it Proof of (\ref{eta11-1}). }
We have
$$
0=\partial \hat{p}_*(\Sigma^9q_3)^*\langle\overline{\nu}_6+\varepsilon_6\rangle
=\partial(\Sigma^9q_3)^*(\overline{\nu}_6+\varepsilon_6)
=\partial(\Sigma^9q_3)^*\overline{\nu}_6 + 2\cdot\partial\Sigma p_*\overline{[\nu_5^2]},
$$
where the last equality follows from (\ref{eta11-7}). 
Hence
$$
-2\cdot\partial\Sigma p_*\overline{[\nu_5^2]}=\partial(\Sigma^9q_3)^*\overline{\nu}_6=(\Sigma^8q_3)^*\partial\overline{\nu}_6
=(\Sigma^8q_3)^*i_*\varepsilon' ,
$$
where the last equality follows from Lemma \ref{eta11-3} (1). 
Thus the order of $\partial\Sigma p_*\overline{[\nu_5^2]}$ is $4$. 
On the other hand, 
$$
(\Sigma^8i')^* (2\cdot\overline{ [\nu_5^2] })=2[\nu_5^2]=\partial\nu_6^2=\partial(\Sigma^9i')^*\Sigma p_*\overline{[\nu_5^2]}=(\Sigma^8i')^*\partial\Sigma p_*\overline{[\nu_5^2]}.
$$
Hence there exists an integer $x$ such that $2\cdot\overline{[\nu_5^2]}-\partial\Sigma p_*\overline{[\nu_5^2]}=x\cdot(\Sigma^8q_3)^*i_*\varepsilon'$. 
Thus $4\cdot\overline{[\nu_5^2]}=2\cdot\partial\Sigma p_*\overline{[\nu_5^2]}=(\Sigma^8q_3)^*i_*\varepsilon'$. 
Therefore the order of $\overline{[\nu_5^2]}$ is $8$, and we obtain (\ref{eta11-1}). 

\medskip

\vspace{0.2cm}
\section{$\pi_n\mathrm{map}_*(\Sp(2), \Sp(2))$}

In this section we compute $\pi_n\mathrm{map}_*(\Sp(2), \Sp(2))$. 
Let $f:\s^9\to \s^3\cup_{\omega}e^7$ be the attaching map of the top cell of $\Sp(2)$, that is, 
$\Sp(2)=\s^3\cup_{\omega}e^7\cup_f e^{10}$. 
The double suspension of $f$ is trivial, that is $\Sigma^2f = 0$, because $\Sigma^2f$ is an element of 
the homotopy group $\pi_{11}(\s^5 \cup_{\Sigma^2\omega}e^9)$ which is isomorphic to the stable group, 
while $f$ is a stably trivial element by \cite{bs}. 
Thus we obtain 
$$\Sigma^2 \Sp(2) \simeq \s^5 \cup_{\Sigma^2\omega}e^9 \vee \s^{12}.$$  
The $p$-components of the homotopy groups for $p \geq 5$ are easily obtained from the results in \cite{T3}, since if $p \geq 5$
$$\Sp(2)_{(p)} \simeq \s^3_{(p)} \times \s^7_{(p)}$$
and thus for $n\ge 1$
\begin{equation}
[\Sigma^n \Sp(2), \Sp(2)]_{(p)}  \cong \big(\pi_{n+3}(\s^3 \times \s^7) \oplus \pi_{n+7}(\s^3 \times \s^7) \oplus \pi_{n+10}(\s^3 \times \s^7)\big)_{(p)}.
\end{equation}
Hence we must compute $2$ and $3$ components of $[\Sigma^n \Sp(2), \Sp(2)]$ for $n\ge 1$. 
The following table shows the generators of 2 and 3 components. Here we use the same notation as before.  

\vspace{0.1cm}
\begin{center}
\begin{tabular}{|c|c|l|}
\hline
$n$ & \text{$2, 3$-{\it components}} & \text{\it generators} \\
\hline\tabtopsp{0.2cm}%
1   & $\mathbb{Z}_2^2  $ & ${\Sigma q}^*i_*\varepsilon_3,\ \overline{\overline{i_*\eta_3}} $ \\[2mm]
\hline\tabtopsp{0.2cm}%
2 & $\mathbb{Z}_2^3  $  & ${\Sigma^2q}^*i_*\mu_3,\ {\Sigma^2q}^*i_*(\eta_3\varepsilon_3),\ \overline{\overline{i_*\eta_3^2}}$\\[2mm]
\hline\tabtopsp{0.2cm}
3 & $\mathbb{Z}_2 \oplus \mathbb{Z}_4 \oplus \mathbb{Z}_8$ & ${\Sigma^3q}^*i_*(\eta_3\mu_4),\ {\Sigma^3q}^*([\nu_7] \nu_{10}),\  \overline{{\Sigma^3q_3}^*[\nu_7]}$ \\[2mm]
\hline\tabtopsp{0.2cm}
4 & $ \mathbb{Z} \oplus \mathbb{Z}_2 \oplus \mathbb{Z}_{16} \oplus \mathbb{Z}_3$ &  $\overline{\overline{3[12\iota_7]}},\  \overline{{\Sigma^4q_3}^*i_*\varepsilon_3},\ {\Sigma^4q}^*[2\sigma'],\ {\Sigma^4q}^*i_*\alpha_3(3) $\\[2mm]
\hline\tabtopsp{0.2cm}
5 & $\mathbb{Z}_2^3 $ & ${\Sigma^5q}^*[\sigma'\eta_{14}],\ \overline{{\Sigma^5q_3}^*i_*\mu_3},\  \overline{{\Sigma^5q_3}^*i_*(\eta_3\varepsilon_4)}$ \\[2mm]
\hline\tabtopsp{0.2cm}
6 & $\mathbb{Z}_2^4  $ & ${\Sigma^6q}^*([\sigma'\eta_{14}]\circ\eta_{15}),\ {\Sigma^6q}^*([\nu_7]\circ\nu_{10}^2),\  \overline{{\Sigma^6q_3}^*([\nu_7]\circ\nu_{10})},\ \overline{{\Sigma^6q_3}^*i_*(\eta_3\mu_4)} $ \\[2mm]
\hline\tabtopsp{0.2cm}
7 & $\mathbb{Z}_8 \oplus \mathbb{Z}_{32} \oplus \bZ_2 \oplus \mathbb{Z}_9  $ & ${\Sigma^7q}^*([\nu_7]\circ\sigma_{10}),\ \overline{\overline{2[\nu_7]}}, \ 2\cdot\overline{\overline{2[\nu_7]}}-z\cdot\overline{{\Sigma^7q_3}^*[2\sigma']},\ \overline{\overline{i_*\alpha_2(3)}}$\\[2mm]
\hline\tabtopsp{0.2cm}
8 & $\mathbb{Z}_2^3 \oplus \mathbb{Z}_8 \oplus \mathbb{Z}_9 $ & ${\Sigma^8q}^*i_*\bar\varepsilon_3,  \overline{\overline{i_*\varepsilon_3}}, \overline{{\Sigma^8q_3}^*[\sigma'\eta_{14}]}, {\Sigma^8q}^*[\zeta_7], {\Sigma^8q}^*[\alpha_3'(7)] $\\[2mm]
\hline
\end{tabular}
\end{center}
\begin{eqnarray*}
\text{Table 3: 2 and 3 components of}\ \pi_n\mathrm{map}_*(\Sp(2), \Sp(2))
\end{eqnarray*}
Here $z$ is an odd integer. 

As in the $\SU(3)$ case, we obtain the following lemma.
\begin{lemma}\label{sp-iso}
$[\Sigma^n \Sp(2), \Sp(2)] \cong \pi_{10+n}(\Sp(2)) \oplus [C_{\Sigma^n\omega}, \Sp(2)]$ for $n \geq 1$.
\end{lemma}
\begin{proof}
The proof is similar to that of Lemma \ref{su-iso}.
\end{proof}

Hence it suffices for our purpose to determine $[C_{\Sigma^n\omega}, \Sp(2)]_{(2,3)}$, 
the $2$ and $3$ components of $[C_{\Sigma^n\omega}, \Sp(2)]$, for $n\ge 1$. 
We use the following results of Mimura-Toda \cite{MT}.
\vspace{0.3cm}
$$
{\small
\begin{array}{|c|c|c||c|c|c|} \hline
n & \pi_n\Sp(2) & \text{gen. of $2,3$-comp.} & n & \pi_n\Sp(2) & \text{gen. of $2,3$-comp.} \\ \hline
1,2,6,8,9 & 0 &  & 12 & \bZ_2\oplus\bZ_2 & i_*\mu_3,\ i_*\eta_3\varepsilon_3\\ \hline
3 & \bZ & i_*\iota_3 & 13 & \bZ_4\oplus\bZ_2 & [\nu_7]\circ\nu_{10},\ i_*\eta_3\mu_4 \\ \hline
4 & \bZ_2 & i_*\eta_3 & 14 & \bZ_{16}\oplus\bZ_3\oplus\bZ_{35} & [2\sigma'],\ i_*\alpha_3(3) \\ \hline
5 & \bZ_2 & i_*\eta_3^2 & 15 & \bZ_2 & [\sigma'\eta_{14}] \\ \hline
7 & \bZ & [12\iota_7] & 16 & \bZ_2\oplus\bZ_2 & [\sigma'\eta_{14}]\circ\eta_{15},\ [\nu_7]\circ\nu_{10}^2 \\ \hline
10 & \bZ_8\oplus\bZ_3\oplus\bZ_5 & [\nu_7],\ i_*\alpha_2(3) & 17 & \bZ_8\oplus\bZ_5 & [\nu_7]\circ\sigma_{10} \\ \hline
11 & \bZ_2 & i_*\varepsilon_3 & 18 & \bZ_8\oplus\bZ_2\oplus\bZ_9\oplus\bZ_{35} & [\zeta_7], i_*\overline{\varepsilon}_3,\  [3\cdot\alpha_3'(7)] \\ \hline
\end{array}
}
$$
\begin{eqnarray*}
\text{Table 4 :}\ \pi_n(\Sp(2))
\end{eqnarray*}

\vspace{0.3cm}
\subsection{$[C_{\Sigma^n\omega}, \Sp(2)]\ (n=1,2)$} 

By the cofibration sequence and Table 4, it is easy to see that
$$
[C_{\Sigma\omega}, \Sp(2)] = \mathbb{Z}_2\big\{i_*\overline{\eta_3}\big\},\quad
[C_{\Sigma^2\omega}, \Sp(2)] = \mathbb{Z}_2\big\{i_*\big(\eta_3\circ\Sigma\overline{\eta_3}\big)\big\}.
$$

\vspace{0.3cm}
\subsection{$[C_{\Sigma^3\omega}, \Sp(2)]$}

By Table 4, we have the following exact sequence.
\begin{equation}\label{sp-pi-3-1}
 \begin{CD}
\bZ\big\{[12\iota_7]\big\} @>\text{$(\Sigma^4\omega)^*$}>> \bZ_8\big\{[\nu_7]\big\}\oplus\bZ_3\{i_*\alpha_2(3)\}\oplus\bZ_5 @>>> [C_{\Sigma^3\omega}, \Sp(2)] @>>> 0. 
 \end{CD}
 \end{equation} 

\begin{lemma}\label{sp-pi-3-2} 
$(\Sigma^4\omega)^*[12\iota_7]= i_*\alpha_2(3).$
\end{lemma}
\begin{proof}
It is known that $\Sigma^4\omega = 2\nu_{7} + \alpha_1(7)$. 
 Let $p : \Sp(2) \to \s^7$ be the bundle projection with fibre $\s^3$. 
Then $p_*([12\iota_7]\circ2\nu_7) = 0$, and hence $[12\iota_7]\circ2\nu_7 = 0$ by Table 4. 
Next consider the composition $[12\iota_7]\circ\alpha_1(7)$. 
We apply Theorem \ref{MT} to the fibration $p:\Sp(2)\to \s^7$ by taking $\alpha=4\iota_7,\ \beta=3\iota_6,\ \gamma=\alpha_1(6)$. 
Then we obtain 
\begin{equation}\label{iotaalpha}
[12\iota_7]\circ\alpha_1(7)= i_*\alpha_2(3).
\end{equation}
Hence $(\Sigma^4\omega)^*[12\iota_7]= i_*\alpha_2(3)$ as desired. 
\end{proof}

Consequently, by (\ref{sp-pi-3-1}) we obtain
$$
[C_{\Sigma^3\omega}, \Sp(2)]_{(2,3)}=\mathbb{Z}_8\big\{(\Sigma^3q_3)^*[\nu_7]\big\}.
$$

\vspace{0.3cm}

\subsection{$[C_{\Sigma^4\omega}, \Sp(2)]$}

By Table 4, we have the following exact sequence.
\begin{equation*}
 \begin{CD}
0 @>>> \bZ_2\{i_*\varepsilon_3\} @>(\Sigma^4q_3)^*>> [C_{\Sigma^4\omega}, \Sp(2)] 
@>(\Sigma^4i')^*>> \bZ\big\{[12\iota_7]\big\} 
@>\text{$\Sigma^4\omega^*$}>> \bZ_{120}
 \end{CD}
 \end{equation*} 
By Lemma \ref{sp-pi-3-2}, $\mathrm{Ker}(\Sigma^4\omega)^* = \bZ\big\{3[12\iota_7]\big\}$. 
It follows that
$$[C_{\Sigma^4\omega}, \Sp(2)] = \mathbb{Z}_2\big\{(\Sigma^4q_3)^*i_*\varepsilon_3\big\} 
\oplus \mathbb{Z}\big\{\overline{3[12\iota_7]}\big\}.$$

\vspace{0.3cm}
\subsection{$[C_{\Sigma^5 \omega}, \Sp(2)]$}

By Table 4, we easily have $(\Sigma^5q_3)^* :\pi_{12}(\Sp(2)) \cong [C_{\Sigma^5\omega}, \Sp(2)]$. 
Hence 
$$
[C_{\Sigma^5\omega},\Sp(2)]=\bZ_2\big\{(\Sigma^5q_3)^*i_*\mu_3\big\}\oplus\bZ_2\big\{(\Sigma^5q_3)^*i_*(\eta_3\varepsilon_4)\big\}.
$$

\vspace{0.3cm}
\subsection{$[C_{\Sigma^6\omega}, \Sp(2)]$}

By Table 4, we have the following exact sequence.
$$
\begin{CD}
\bZ_8\big\{[\nu_7]\big\}\oplus\bZ_{15} @>(\Sigma^7\omega)^*>> \bZ_4\big\{[\nu_7]\circ\nu_{10}\big\}\oplus\bZ_2\{i_*\eta_3\mu_4\} 
@>(\Sigma^6q_3)^*>> [C_{\Sigma^6\omega}, \Sp(2)] @>>> 0
\end{CD}
$$
Hence we obtain
 $$[C_{\Sigma^6\omega}, \Sp(2)] = \mathbb{Z}_2\big\{(\Sigma^6q_3)^*[\nu_7]\circ\nu_{10}\big\} \oplus \mathbb{Z}_2\big\{(\Sigma^6q_3)^*i_*(\eta_3\mu_4)\big\}.$$

\vspace{0.3cm} 
\subsection{$[C_{\Sigma^7\omega}, \Sp(2)]$}

By Table 4, we have the following exact sequence:
$$
0\to\bZ_{16}\big\{[2\sigma']\big\}\oplus\bZ_3\{i_*\alpha_3(3)\}\overset{(\Sigma^7q_3)^*}{\to} [C_{\Sigma^7\omega},\Sp(2)]_{(2,3)}\overset{(\Sigma^7i')^*}{\to} \bZ_4\big\{2[\nu_7]\big\}\oplus\bZ_3\{i_*\alpha_2(3)\}\to 0.
$$
We shall prove
\begin{gather}
[C_{\Sigma^7\omega},\Sp(2)]_{(2)}=\bZ_{32}\Big\{\overline{2[\nu_7]}\Big\}\oplus\bZ_2\Big\{2\cdot\overline{2[\nu_7]}-z\cdot(\Sigma^7q_3)^*[2\sigma']\Big\},\ z\equiv 1\pmod{2},\label{pi7-1}\\
[C_{\Sigma^7\omega},\Sp(2)]_{(3)}=\bZ_{9}\big\{\overline{i_*\alpha_2(3)}\big\}.\label{pi7-2}
\end{gather}

Firstly we prove (\ref{pi7-1}). 
By Table 4 and \cite{T3}, we have the following commutative diagram with exact rows and columns:
\begin{equation}\label{pi7-3}
\xymatrix{
&  \bZ_{16}\big\{[2\sigma']\big\} \ \ar@{>->}[r]^-{q^*} \ar[d]^-{p_*} & [C_{\Sigma^7\omega},\Sp(2)]_{(2)} \ 
\ar@{->>}[r]^-{i^*} \ar[d]^-{p_*} & \bZ_4\big\{2[\nu_7]\big\} \ar[d]^-{p_*}\\
&  \bZ_8\{\sigma'\} \ \ar@{>->}[r]^{q^*} \ar[d]^-{\partial} & [C_{\Sigma^7\omega},\s^7]_{(2)} 
\ \ar @{->>}[r]^-{i^*} \ar[d]^-{\partial} & \bZ_8\{\nu_7\}\\
 & \bZ_4\{\varepsilon'\}\oplus\bZ_2\{\eta_3\mu_4\} \ \ar@{>->>}[r]^-{q^*} \ar[d]^-{i_*} 
& [C_{\Sigma^6\omega},\s^3]_{(2)} \ar[d]^-{i_*}  \\
\bZ_8\big\{[\nu_7]\big\} \ar[r]^-{(2\nu_{10})^*} & \bZ_4\big\{[\nu_7]\circ\nu_{10}\big\}\oplus\bZ_2\{i_*\eta_3\mu_4\} 
\ar[r]^-{q^*} & [C_{\Sigma^6\omega},\Sp(2)]_{(2)} 
}
\end{equation}
We claim that the second row splits:
\begin{equation}\label{pi7-4}
[C_{\Sigma^7\omega},\s^7]_{(2)}=\bZ_8\{q^*\sigma'\}\oplus\bZ_8\{\overline{\nu_7}\}.
\end{equation}
This is done as follows. 
By \cite{T3}, we easily have
\begin{equation}\label{pi7-5}
[C_{\Sigma^3\omega},\s^3]_{(2)}=\bZ_4\big\{\overline{\nu'}\big\}
\end{equation}
and the following exact sequence:
$$
\begin{CD}
0@>>>\bZ_2\{\sigma'''\}@>q^*>>[C_{\Sigma^5\omega},\s^5]_{(2)}@>i^*>>\bZ_8\{\nu_5\}@>>>0.
\end{CD}
$$
Since $i^*\big(2\cdot\overline{\nu_5}-\Sigma^2\overline{\nu'}\big)=0$, we can write 
$2\cdot\overline{\nu_5}-\Sigma^2\overline{\nu'}=c\cdot q^*\sigma'''$ $(c\in\bZ)$. 
Then $4\cdot\overline{\nu_5}-2\cdot\Sigma^2\overline{\nu'}=0$ so that the order of $\overline{\nu_5}$ is $8$, 
since $i^*(2\cdot\Sigma^2\overline{\nu'})=4\nu_5$ so that the order of $2\cdot\Sigma^2\overline{\nu'}$ is $2$ by (\ref{pi7-5}). 
Define $\overline{\nu_7}:=\Sigma^2\overline{\nu_5}$. 
Then the order of $\overline{\nu_7}$ is $8$, for the order of $i^*(\overline{\nu_7})=\nu_7$ is $8$. 
Thus we obtain (\ref{pi7-4}). 

In (\ref{pi7-3}), we have $i^*\varepsilon'=2[\nu_7]\circ\nu_{10}=(\Sigma^7\omega)^*[\nu_7]$ by \cite{MT}. 
Hence $\partial\sigma'=2\varepsilon'$, $i_*q^*\varepsilon'=q^*i_*\varepsilon'=0$ and
\begin{equation}\label{pi7-6}
\partial q^*\sigma'=q^*\partial\sigma'=2 q^*\varepsilon'.
\end{equation}
Hence the kernel of the second $i_*$ of (\ref{pi7-3}) equals to $\bZ_4\{q^*\varepsilon'\}$. 
This kernel equals to the image of the second $\partial$ of (\ref{pi7-3}). 
Hence 
\begin{equation}\label{pi7-7}
\partial\overline{\nu_7}=\pm q^*\varepsilon'
\end{equation}
by (\ref{pi7-4}) and (\ref{pi7-6}). 
We have $i^*\big(2\cdot\overline{\nu_7}-p_*\overline{2[\nu_7]}\big)=0$ so that we can write
\begin{equation}\label{pi7-8}
2\cdot\overline{\nu_7}-p_*\overline{2[\nu_7]}=a\cdot q^*\sigma'\quad(a\in\bZ).
\end{equation}
We then have
\begin{align*}
2a\cdot q^*\varepsilon'&=\partial(a\cdot q^*\sigma')\quad(\text{by (\ref{pi7-6})})\\
&=\partial\big(2\cdot\overline{\nu_7}- p_*\overline{2[\nu_7]}\big)=2\cdot\partial\overline{\nu_7}\\
&=2\cdot q^*\varepsilon'\quad(\text{by (\ref{pi7-7}}).
\end{align*}
Hence $2a\equiv 2\pmod{4}$, that is, $a$ is odd. 
It follows that, by multiplying $4$ with (\ref{pi7-8}), we have
$$
4\cdot q^*\sigma'=-4\cdot p_*\overline{2[\nu_7]}.
$$
On the other hand, we can write
\begin{equation}\label{pi7-9}
4\cdot\overline{2[\nu_7]}=y\cdot q^*[2\sigma']\quad (y\in\bZ).
\end{equation}
Hence we have
$$
4\cdot q^*\sigma'=-y\cdot p_*q^*[2\sigma']=-2y\cdot q^*\sigma'.
$$
Hence $-2y\equiv 4\pmod{8}$, that is, 
\begin{equation}\label{pi7-10}
y\equiv 2\pmod{4}.
\end{equation}
Thus the order of $4\cdot\overline{2[\nu_7]}$ is $8$, that is, the order of $\overline{2[\nu_7]}$ is $32$. 
Also the order of $2\cdot\overline{2[\nu_7]}-(y/2)\cdot q^*[2\sigma']$ is $2$. 
Therefore we obtain (\ref{pi7-1}) by the first row of (\ref{pi7-3}). 

As a byproduct of (\ref{pi7-10}), we have

\begin{corollary}
$[\nu_7]\circ\eta_{10}=i_*\varepsilon_3\in\pi_{11}(\Sp(2))=\bZ_2\{i_*\varepsilon_3\}$.
\end{corollary}
\begin{proof}
Since indeterminacy of $\{2[\nu_7],2\nu_{10},4\iota_{13}\}$ is $4\cdot\pi_{14}(\Sp(2))$, 
we can write
\begin{equation}\label{pi7-11}
\big\{2[\nu_7],2\nu_{10},4\iota_{13}\big\}=x\cdot[2\sigma']+4\cdot\pi_{14}(\Sp(2)).
\end{equation}
Let $\psi^k:\Sp(2)\to\Sp(2)$ be defined by $\psi^k(A)=A^k$. 
We have
$$
\psi^2\circ\big\{2[\nu_7], 2\nu_{10}, 4\iota_{13}\big\}\subset \big\{4[\nu_7],2\nu_{10},4\iota_{13}\big\}
\subset\big\{[\nu_7],8\nu_{10},4\iota_{13}\big\}=4\pi_{14}(\Sp(2)).
$$
Hence $2x[2\sigma']\in 4\pi_{14}(\Sp(2))=\bZ_{4}\big\{4[2\sigma']\big\}\oplus\bZ_{105}$ by Table 4. 
Thus $x\equiv 0\pmod{2}$. 
On the other hand
\begin{equation}\label{pi7-12}
\big\{2[\nu_7], 2\nu_{10}, 4\iota_{13}\big\}=\big\{[\nu_7],4\nu_{10},4\iota_{13}\big\}=\big\{[\nu_7],\eta_{10}^3,4\iota_{13}\big\}
=\{[\nu_7]\circ\eta_{10},\eta_{11}^2,4\iota_{13}\}.
\end{equation}
To induce a contradiction, assume $[\nu_7]\circ\eta_{10}=0$. 
Then $\big\{2[\nu_7],2\nu_{10},4\iota_{13}\big\}=4\pi_{14}(\Sp(2))$ by (\ref{pi7-12}) and $x\equiv 0\pmod{4}$ by (\ref{pi7-11}). 
We then have
\begin{align*}
0&=4\cdot\big(\overline{2[\nu_7]}\circ\widetilde{4\iota_{13}}\big)
=\psi^4\circ\overline{2[\nu_7]}\circ\widetilde{4\iota_{13}}
=\big(4\cdot\overline{2[\nu_7]}\big)\circ\widetilde{4\iota_{13}}\\
&=\big(y\cdot q^*[2\sigma']\big)\circ\widetilde{4\iota_{13}}\quad(\text{by (\ref{pi7-9})})\\
&=\psi^y\circ[2\sigma']\circ q\circ\widetilde{4\iota_{13}}=\psi^y\circ[2\sigma']\circ 4\iota_{14}\\
&=4y[2\sigma']
\end{align*}
Thus $4y\equiv 0\pmod{16}$, that is, $y\equiv 0\pmod{4}$. 
This contradicts (\ref{pi7-10}). 
\end{proof}

Next we consider the 3-primary part of $[C_{\Sigma^7\omega}, \Sp(2)]$, that is, we prove (\ref{pi7-2}). 
First we remark that
  $$[C_{\Sigma^7\omega}, \Sp(2)]_{(3)} \cong [C_{\alpha_1(10)}, \Sp(2)]_{(3)}.$$
Hence it suffices to prove
$$
[C_{\alpha_1(10)}, \Sp(2)]_{(3)}\cong\bZ_9.
$$
We shall prove this as follows. 

\begin{prop}\label{os-lemma}
\begin{enumerate}
\item $\{\alpha_1(5), \alpha_1(8), 3\iota_{11}\}=\{3\iota_5, \alpha_1(5), \alpha_1(8)\}
=2\alpha_2(5)+3\pi_{12}(\s^5)\label{os-lemma(1)}$.
\item $i\circ\alpha_3(3)=[12\iota_7]\circ\alpha_2(7)\in\pi_{10}(\Sp(2))$.
\item $[C_{\alpha_1(10)},\Sp(2)]_{(3)}\cong[C_{\alpha_1(10)},\s^7]_{(3)}$.
\item $[C_{\alpha_1(10)},\s^7]_{(3)}\cong\bZ_9$. 
\end{enumerate}
\end{prop}

\medskip

\noindent{\it Proof of Proposition \ref{os-lemma} (1). }
It follows from \cite[Proposition 1.3]{T3} that 
\begin{gather*}
\Sigma^\infty\{\alpha_1(5), \alpha_1(8), 3\iota_{11}\}\subset\langle\alpha_1,\alpha_1,3\rangle,\quad 
\Sigma^\infty\{3\iota_5, \alpha_1(5), \alpha_1(8)\}\subset\langle 3,\alpha_1,\alpha_1\rangle,\\
\Sigma^\infty\{\alpha_1(3), 3\iota_6, \alpha_2(6)\}\subset\langle\alpha_1,3,\alpha_2\rangle .
\end{gather*}
We use following relations \cite[(3.9)]{T3}:
\begin{equation}\label{os-equation1}
\begin{gathered}
\langle\alpha_1,\alpha_1,3\rangle-\langle\alpha_1,3,\alpha_1\rangle+\langle 3,\alpha_1,\alpha_1\rangle\ni 0,\\
\langle\alpha_1,\alpha_1,3\rangle=\langle 3,\alpha_1,\alpha_1\rangle.
\end{gathered}
\end{equation}
Let $A\in\langle\alpha_1,\alpha_1,3\rangle$. 
Since $\langle\alpha_1,3,\alpha_1\rangle=\alpha_2$ and 
$\mathrm{Indet}\langle\alpha_1,\alpha_1,3\rangle=3 {G_{7}}$, 
it follows from (\ref{os-equation1}) that $2A - \alpha_2+3G_7\ni 0$ so that $A\in 2\alpha_2+3G_7$, 
since ${G_7}_{(3)} =\mathbb{Z}_3\{\alpha_2\}$, 
where $G_k$ denotes the $k$-th stable homotopy group of the sphere. 
Hence $\langle\alpha_1,\alpha_1,3\rangle=2\alpha_2+3G_7$. 
Since $\Sigma^\infty:\pi_{12}(\s^5)=\mathbb{Z}_3\{\alpha_2(5)\}\oplus\bZ_{10} \to G_7$ is injective 
and $\mathrm{Indet}\{\alpha_1(5),\alpha_1(8),3\iota_{11}\}
=\mathrm{Indet}\{3\iota_5,\alpha_1(5),\alpha_1(8)\}=3\pi_{12}(\s^5) $, it follows that 
$$
2\alpha_2(5)\in\{\alpha_1(5),\alpha_1(8),3\iota_{11}\}\cap\{3\iota_5, \alpha_1(5),\alpha_1(8)\}
$$ 
so that 
$\{\alpha_1(5),\alpha_1(8),3\iota_{11}\}=\{3\iota_5, \alpha_1(5),\alpha_1(8)\}=2\alpha_2(5)+3\pi_{12}(\s^5)$. 
\qed

\medskip

\noindent{\it Proof of Proposition \ref{os-lemma} (2). }
We can apply Theorem \ref{MT} to the fibration $\Sp(2)\to\s^7$ by taking 
$\alpha=4\iota_7$, $\beta=3\iota_6$ and $\gamma=\alpha_2(6)$. 
Indeed, we have $\beta\circ\gamma=0$ and $\partial\alpha\circ\beta=\alpha_1(3)\circ 3\iota_6=0$ 
since $\partial\iota_7=\omega=\nu'+\alpha_1(3)$. 
Hence we can use Theorem~\ref{MT} in this case. 
Therefore there exists $\epsilon\in\pi_7(\Sp(2))$ such that $p_*\epsilon=12\iota_7$ and $i_*(\alpha_3(3))=\epsilon\circ\alpha_2(7)$ 
so that $\epsilon=[12\iota_7]$ and $i_*(\alpha_3(3))=[12\iota_7]\circ\alpha_2(7)$.
\qed

\medskip

\noindent{\it Proof of Proposition \ref{os-lemma} (3). }
By \cite{T3} and Table 4, we have the following commutative diagram with exact rows.
$$
\begin{CD}
0 @>>>\bZ_3\{i_*\alpha_3(3)\} @>>>[C_{\alpha_1(10)}, \Sp(2)]_{(3)} @>>>
\bZ_3\{i_*\alpha_2(3)\} @>>> 0\\
@. @AA[12\iota_7]_*A @AA[12\iota_7]_*A @AA[12\iota_7]_*A @.\\
0 @>>>\bZ_3\{\alpha_2(7)\} @>>>[C_{\alpha_1(10)}, \s^7]_{(3)} @>>>
\bZ_3\{\alpha_1(7)\} @>>>0.
\end{CD}
$$
It follows from (\ref{iotaalpha}) and Proposition~\ref{os-lemma}~(2) 
that the first and the third $[12\iota_7]_*$ are isomorphisms 
so that the second $[12\iota_7]_*$ is also an isomorphism. 
Hence we obtain Proposition~\ref{os-lemma}~(3).
\qed

\medskip

\noindent{\it Proof of Proposition \ref{os-lemma} (4). }
We shall prove the following:
$$
[C_{\alpha_1(10)},\s^7]_{(3)}\overset{\Sigma}{\cong}[C_{\alpha_1(9)},\s^6]_{(3)}\overset{\Sigma}{\cong}[C_{\alpha_1(8)},\s^5]_{(3)}=\bZ_9\big\{\overline{\alpha_1(5)}\big\}.
$$
By \cite{T3} and the fact $\alpha_1(5)\circ\alpha_1(8)=0$ (\cite[(13.7)]{T3}), we have the following commutative diagram with exact rows.
$$
\begin{CD}
0@>>>\bZ_3\{\alpha_2(5)\}\oplus\bZ_{10}@>{\Sigma^5 q'}^*>>[C_{\alpha_1(8)},\s^5]
@>{\Sigma^5 i''}^*>>\bZ_3\{\alpha_1(5)\}\oplus\bZ_8@>>>0\\
@. @VV\Sigma V @VV\Sigma V @V\cong V\Sigma V @.\\
0@>>>\bZ_3\{\alpha_2(6)\}\oplus\bZ_{20}@>{\Sigma^6 q'}^*>>[C_{\alpha_1(9)},\s^6]
@>{\Sigma^6 i''}^*>>\bZ_3\{\alpha_1(6)\}\oplus\bZ_8@>>>0\\
@. @VV\Sigma V @VV\Sigma V @V\cong V\Sigma V @.\\
0@>>>\bZ_3\{\alpha_2(7)\}\oplus\bZ_{40}@>{\Sigma^7 q'}^*>>[C_{\alpha_1(10)},\s^7]
@>{\Sigma^7 i''}^*>>\bZ_3\{\alpha_1(7)\}\oplus\bZ_8@>>>0
\end{CD}
$$
Here $q':C_{\alpha_1(3)}\to\s^7$ is the quotient and $i'':\s^3\to C_{\alpha_1(3)}$ is the inclusion. 
By the EHP-sequence (\cite[(2.11)]{T3}), we know that two $\Sigma$'s in the first column are monomorphisms. 
Hence two $\Sigma$'s in the second column are also monomorphisms. 
Thus suspensions induce 
$$
[C_{\alpha_1(8)},\s^5]_{(3)}\cong [C_{\alpha_1(9)},\s^6]_{(3)}\cong [C_{\alpha_1(10)},\s^7]_{(3)}.
$$
Since $\Sigma\big(3\,\overline{\alpha_1(5)}\big)=\Sigma\big(3\iota_5\circ\overline{\alpha_1(5)}\big)$, it follows that 
$3\,\overline{\alpha_1(5)}=3\iota_5\circ\overline{\alpha_1(5)}$. 
We have
\begin{align*}
3\iota_5\circ\overline{\alpha_1(5)}\in\{3\iota_5, &\alpha_1(5), \alpha_1(8)\}\circ \Sigma^5 q' \qquad (\text{by \cite[Proposition 1.9]{T3}})\\
&=\big( 2\alpha_2(5)+3\pi_{12} (\s^5) \big) \circ \Sigma^5 q'\qquad (\text{by Proposition \ref{os-lemma} (1)})
\end{align*}
Hence we can write
$$
3\,\overline{\alpha_1(5)}=3\iota_5\circ\overline{\alpha_1(5)}={\Sigma^5 q'}^*\big(2\alpha_2(5)+x\big),\quad 10 x=0.
$$
Thus the order of $\overline{\alpha_1(5)}$ is a multiple of $9$. 
Therefore $[C_{\alpha_1(8)},\s^5]_{(3)}=\bZ_9\big\{\overline{\alpha_1(5)}\big\}$. 
This completes the proof of Proposition \ref{os-lemma}. 
\qed

\vspace{0.3cm}
\subsection{$[C_{\Sigma^8\omega}, \Sp(2)]$}

Since $\Sigma^m\omega=2\nu_{m+3}+\alpha_1(m+3)$ for $m\ge 2$, 
we have 
$$
(\Sigma^9\omega)^*\pi_{12}(\Sp(2))=0,\quad (\Sigma^8\omega)^*\pi_{11}(\Sp(2))=0
$$ 
by Table 4.  
Hence we have the following commutative diagram with exact rows.
$$
\begin{CD}
0@>>>\bZ_2\big\{[\sigma'\eta_{14}]\big\} @>(\Sigma^8q_3)^*>>[C_{\Sigma^8\omega},\Sp(2)]@>(\Sigma^8 i')^*>>\bZ_2\{i_*\varepsilon_3\} @>>>0\\
@.  @VVV @VVV @VVV @.\\
0@>>>\bZ_2\{\sigma'\eta_{14}\}\oplus\bZ_2^2 @>(\Sigma^8q_3)^*>\cong>[C_{\Sigma^8\omega},\s^7]@>>>0 @.
\end{CD}
$$
Thus we easily have 
$$
[C_{\Sigma^8\omega},\Sp(2)]=\bZ_2\big\{(\Sigma^8q_3)^*[\sigma'\eta_{14}]\big\}\oplus\bZ_2\big\{i_*\overline{\varepsilon_3}\big\}.
$$

\vspace{0.3cm}
\section{$\pi_1\mathrm{map}_*(G_2, G_2)$}

In this section we shall compute $[\Sigma G_2, G_2]\ (\cong \pi_1\mathrm{map}_*(G_2, G_2))$. 
As in the subsection 3.7, we use the fibration
$$
\begin{CD}
\SU(3)@>\hat{i}>> G_2 @>\hat{p}>>\s^6 ,
\end{CD}
$$
and the following results from \cite{M}.
\begin{center}
\begin{tabular}{|c|c|c|} \hline
$n$ & $\pi_nG_2$ & \mbox{gen. of 2-comp.} \\ \hline
1,2,4,5,7,10,12,13 & 0 &  \\[1mm] \hline\tabtopsp{0.1cm}
3 & $\bZ$ & $\hat{i}_*\iota_3$  \\[1mm] \hline\tabtopsp{0.1cm}
6 & $\bZ_3$ &   \\[1mm] \hline\tabtopsp{0.1cm}
8 & $\bZ_2$ & $\langle\eta_6^2\rangle$  \\[1mm] \hline\tabtopsp{0.1cm}
9 & $\bZ_6$ & $\langle\eta_6^2\rangle\circ\eta_8$\\[1mm] \hline\tabtopsp{0.1cm}
11 & $\bZ \oplus \bZ_2$ & $\langle2\Delta\iota_{13}\rangle, \hat{i}_*[\nu_5^2]$ \\[1mm] \hline\tabtopsp{0.1cm}
14 & $\bZ_{168}\oplus \bZ_2$ & $\langle \bar{\nu}_6 + \epsilon_6\rangle,  \hat{i}_*[\nu_5^2]\circ\nu_{11}$ \\[1mm] \hline\tabtopsp{0.1cm}
15 & $\bZ_2$ & $\langle \bar{\nu}_6 + \epsilon_6\rangle\circ\eta_{14}$ \\[1mm] \hline
\end{tabular}
\end{center}
\begin{eqnarray*}
\text{Table 5 :}\ \pi_n(G_2)
\end{eqnarray*}
In the Table~5 we follow the notations in \cite{M}.\\[0.3cm]
As is well-known, $G_2$ has the cell structure:
$$G_2 = \s^3 \cup e^5 \cup e^6 \cup e^8 \cup e^9 \cup e^{11} \cup e^{14}.$$
Let $G_2^{(n)}$ denote the $n$-skeleton of $G_2$. Let $M^n=C_{2\iota_{n-1}}=\s^{n-1}\cup_{2\iota_{n-1}} e^n$ for $n\ge 2$,  and 
$$
\begin{CD}
\s^{n-1}@>i_{n}>> M^n@>q_{n}>>\s^n
\end{CD}
$$
be the inclusion and the quotient map, respectively. 
Remark that $\Sigma M^n=M^{n+1}$. 
Then there exist the cofibrations as follows.
\begin{gather}
 \label{G2-1}\s^3 \to G_2^{(6)} \xrightarrow{\pi_1}  M^6,\\
\label{G2-2}G_2^{(6)} \to G_2^{(9)} \xrightarrow{\pi_2}  M^9 \xrightarrow{\delta} \Sigma G_2^{(6)}.  
\end{gather}
From (\ref{G2-1}) we obtain \cite[Lemma 3.6]{MS}:

\begin{lemma}[\cite{MS}]\label{G2-lemma1} 
$[\Sigma G_2^{(6)}, G_2] = 0$.
\end{lemma}

Next we shall show the following. 

\begin{lemma}\label{G2-lemma2}   
$ {\Sigma\pi_2}^*: [M^{10}, G_2] \to [\Sigma G_2^{(9)}, G_2] $ 
is an isomorphism. 
\end{lemma} 
\begin{proof}
From Lemma \ref{G2-lemma1} it suffices to show that $(\Sigma\delta)^*:[\Sigma^2 G_2^{(6)}, G_2]\to [\Sigma M^9, G_2]$ is trivial. 
By Table 5 we easily have
\begin{equation}\label{M2G2}
[\Sigma M^9, G_2]=\bZ_2\{\langle\eta_6^2\rangle\circ\overline{\eta_8}\},\quad (\Sigma i_9)^*(\langle\eta_6^2\rangle\circ\overline{\eta_8})=\langle\eta_6^2\rangle\circ\eta_8
\end{equation}
and
$$
\begin{CD}
\pi_8(G_2)@>(\Sigma^2 q_6)^*>\cong> [\Sigma^2 M_6, G_2]@>(\Sigma^2\pi_1)^*>\cong> [\Sigma^2 G_2^{(6)}, G_2].
\end{CD}
$$
Hence it suffices to to prove the following equality:
$$
(\Sigma i_9)^*(\Sigma\delta)^*(\Sigma^2\pi_1)^*(\Sigma^2 q_6)^*\langle\eta_6^2\rangle=0.
$$
We shall prove this by showing
\begin{equation}\label{1stem}
\Sigma^2 q_6\circ\Sigma^2\pi_1\circ\Sigma\delta\circ\Sigma i_9=0\in\pi_9(\s^8)=\bZ_2\{\eta_8\}.
\end{equation}
By \cite{Mu}, we have the following results.
\begin{gather}
\label{Moore1}[M^{10}, S^8] =\bZ_4\{\overline{\eta_8}\},\quad 2\overline{\eta_8} = \eta_8^2\circ q_{10},\\
\label{Moore2}[M^{10}, M^8] \cong \bZ_2^3.
\end{gather}
We have $2(\Sigma^2\pi_1\circ\Sigma\delta)=0$ by (\ref{Moore2}). 
Hence it follows from (\ref{Moore1}) that $\Sigma^2 q_6\circ\Sigma^2\pi_1\circ\Sigma\delta$ is divisible by $2$. 
Thus (\ref{1stem}) is established. 
\end{proof}

Next we shall show that 

\begin{lemma}\label{G2-lemma3} 
\begin{enumerate}
\item The induced map
$$\Sigma i_{9, 11}^* : [\Sigma G_2^{(11)}, G_2] \to [\Sigma G_2^{(9)}, G_2]$$
is an isomorphism, where $i_{9, 11} : G_2^{(9)} \to G_2^{(11)}$ is the inclusion.
\item $[\Sigma G_2^{(11)}, G_2]=\bZ_2\Big\{\overline{\overline{\langle\eta_6^2\rangle\circ\overline{\eta_8}\circ\Sigma\pi_2}}\Big\}$.
\end{enumerate}
\end{lemma} 
\begin{proof}
The assertion (1) follows from $\pi_{12}(G_2) = 0$ (\cite{M}) and  \cite[Lemmas 3.9\,(i) and 3.11]{MS} using the cofibration 
$$
\begin{CD}
 \s^{10}@>>>G_2^{(9)} @>i_{9,11}>> G_2^{(11)}.
\end{CD}
$$
The assertion (2) follows from (1), (\ref{M2G2}) and Lemma \ref{G2-lemma2}.
\end{proof}

Let $f : \s^{13} \to G_2^{(11)}$ denote the attaching map of the top cell of $G_2$. 

\begin{lemma}\label{G2-lemma4} There exists the following short exact sequence.
\begin{equation}\label{G2-short-exact}
\begin{CD}
0 @>>> \mathbb{Z}_2 @>>> [\Sigma G_2, G_2] @>>> \mathbb{Z}_2 @>>> 0
\end{CD}
\end{equation}
\end{lemma}
\begin{proof}
In the exact sequence induced by the cofibration  $\s^{13} \xrightarrow{f} G_2^{(11)} \subset G_2$
\begin{equation}\label{G2-6}
[\Sigma^2G^{(11)}, G_2] \xrightarrow{(\Sigma^2f)^*} \pi_{15}(G_2) \xrightarrow{(\Sigma q)^*} [\Sigma G_2, G_2] 
\longrightarrow [\Sigma G^{(11)}_2, G_2] \xrightarrow{(\Sigma f)^*} \pi_{14}(G_2)
\end{equation}
$(\Sigma f)^*$ is trivial by \cite[Lemma 3.13]{MS}. 
Here $q:G_2\to\s^{14}$ is the quotient map. 
We show that $(\Sigma^2 f)^*$ is also trivial. 
To prove this, first we recall that
$$\pi_{15}(G_2) = \mathbb{Z}_2\big\{\langle\bar\nu_6 + \varepsilon_6\rangle\circ \eta_{14}\big\}$$ from \cite{M}. 
Here $\langle\bar\nu_6 + \varepsilon_6\rangle$ is an element of $\pi_{14}(G_2)$ 
such that $\hat{p}_*\langle\bar\nu_6 + \varepsilon_6\rangle = \bar\nu_6 + \varepsilon_6$ 
by the bundle projection map $\hat{p}: G_2 \to \s^6$. 
By \cite[Lemma 6.3, Theorem 7.2]{T3},  $(\bar\nu_6 + \varepsilon_6)\circ \eta_{14} $ is stably nontrivial 
and so is $\langle\bar\nu_6 + \varepsilon_6\rangle\circ\eta_{14}$. 
On the other hand, the attaching map $f$ is stably trivial by \cite{bs}. 
This means 
$$\mathrm{Im}~(\Sigma^2f)^* = 0$$
in (\ref{G2-6}). 
Thus by (\ref{G2-6}), Lemma \ref{G2-lemma2} and Lemma \ref{G2-lemma3}, we obtain the result. 
\end{proof}

\begin{theorem}\label{G2-theorem}
$$ [\Sigma G_2, G_2] =\bZ_2\{\langle\overline{\nu}_6+\varepsilon_6\rangle\circ\eta_{14}\circ\Sigma q\}\oplus
\bZ_2\Big\{\overline{\overline{\langle\eta_6^2\rangle\circ\overline{\eta_8}\circ\Sigma\pi_2}}\Big\}. $$
\end{theorem}

\begin{proof}
By Lemma \ref{G2-lemma4}, $[\Sigma G_2,G_2]$ is isomorphic to $\bZ_2^2$ or $\bZ_4$. 
To induce a contradiction, assume that it is isomorphic to $\bZ_4$. 
In this case, by Lemma \ref{G2-lemma3}\,(2) and the proof of Lemma \ref{G2-lemma4}, we have
$$
2\, \overline{\overline{\langle \eta_6^2\rangle \circ \overline{\eta_8} \circ \Sigma\pi_2}} 
=\langle \bar{\nu}_6 + \epsilon_6 \rangle \circ \eta_{14} \circ \Sigma q.
$$
Let $\ell : \{\Sigma G_2, G_2\} \to  \pi^s_{15}(G_2)$ be a left inverse for $\Sigma^\infty q^* : \pi^s_{15}(G_2) \to \{\Sigma G_2, G_2\}$. 
It exists, because $\Sigma^\infty f=0$. 
Here $\{X,Y\}=\lim_{n\to\infty}[\Sigma^nX,\Sigma^nY]$ and $\pi_n^s(X)=\{\s^n,X\}$. 
We then have
\begin{align*}
2\,\Sigma^\infty\hat{p}_* \circ \ell \Big(\Sigma^\infty \overline{\overline{\langle \eta_6^2\rangle \circ \overline{\eta_8} \circ \Sigma\pi_2}}\Big) &=  \Sigma^\infty \hat{p}_* \circ \ell \Big(2\Sigma^\infty \overline{\overline{\langle \eta_6^2\rangle \circ \overline{\eta_8} \circ \Sigma\pi_2}}\Big)\\
&=  \Sigma^\infty \hat{p}_* \circ \ell \circ \Sigma^\infty q^*(\langle \bar{\nu} + \epsilon \rangle \circ \eta)\\
&= (\bar{\nu} + \epsilon)\eta\\
&= \eta^2\sigma.
\end{align*}
Note that the element 
$2\Sigma^\infty\hat{p}_* \circ \ell \Big(\Sigma^\infty \overline{\overline{\langle \eta_6^2\rangle \circ \overline{\eta_8} \circ \Sigma\pi_2}}\Big)$
is trivial since $\pi_9^s(S^0) \cong \bZ_2^3$ (\cite{T3}). This contradicts  $\eta^2\sigma \not= 0$ (\cite{T3}).
Therefore, the short exact sequence (\ref{G2-short-exact}) splits and we obtain the result. 
\end{proof}

\end{document}